\documentclass{amsart}
\usepackage[psamsfonts]{amssymb}
\usepackage{amsmath,amsthm}

\usepackage{chngcntr}
\usepackage{apptools}
\AtAppendix{\counterwithin{proposition}{section}}
\newtheorem{proposition}{Proposition}

\usepackage{enumerate}
\usepackage{multirow}

\newtheorem{thm}{Theorem}[section]
\newtheorem{cor}[thm]{Corollary}
\newtheorem{lem}[thm]{Lemma}
\newtheorem{prop}[thm]{Proposition}

\theoremstyle{remark}
\newtheorem{rem}{Remark}[section]

\def\sph{\mathbb{S}^{d-1}}
\def\f{\frac}

 \def\a{{\alpha}} 
 \def\b{{\beta}}
 
 \def\k{{\kappa}}
 \def\t{{\theta}}
 \def\l{{\lambda}}

 \def\s{{\sigma}}
 \def\la{{\langle}}
 \def\ra{{\rangle}}

 \def\mb{{\mathbf m}}

 \def\CH{{\mathcal H}}

 \def\CP{{\mathcal P}}

 \def\CV{{\mathcal V}}
 
 \def\BB{{\mathbb B}}

 \def\NN{{\mathbb N}}

 \def\RR{{\mathbb R}}
  \def\SS{{\mathbb S}}

        \def\proj{\operatorname{proj}}

\newcommand{\wh}{\widehat}

\begin{document}
 
\title[Best Polynomial Approximation on the Unit Ball]
{Best Polynomial Approximation on the Unit Ball}

\author[M. A. Pi\~{n}ar]{Miguel A. Pi\~{n}ar}
\address[M. A. Pi\~{n}ar]{{Departamento de Matem\'atica Aplicada\\
Universidad de Granada\\
18071 Granada, Spain}}\email{mpinar@ugr.es}

\author[Yuan Xu]{Yuan Xu}
\address[Yuan Xu]{Department of Mathematics\\ University of Oregon\\
    Eugene, Oregon 97403-1222.}\email{yuan@math.uoregon.edu}

\thanks{The first author thanks MINECO of Spain and the European Regional Development Fund (ERDF) 
for their support through the grant MTM2014--53171--P and Jun\-ta de Andaluc\'{\i}a research group FQM--384.
The second author was supported in part by NSF Grant DMS-1510296.}

\date{\today}

\begin{abstract} 
Let $E_n(f)_\mu$ be the error of best approximation by polynomials of degree at 
most $n$ in the space $L^2(\varpi_\mu, \BB^d)$, where $\BB^d$ is the unit ball in $\RR^d$ and
$\varpi_\mu(x) = (1-\|x\|^2)^\mu$ for $\mu > -1$. Our main result shows that, for $s \in \NN$,  
$$
E_n(f)_\mu \le c n^{-2s}[E_{n-2s}(\Delta^s f)_{\mu+2s} + E_{n}(\Delta_0^s f)_{\mu}],
$$
where $\Delta$ and $\Delta_0$ are the Laplace and Laplace-Beltrami operators, respectively. 
We also derive a bound when the right hand side contains odd order derivatives. 
\end{abstract}

\maketitle

\section{Introduction and Main Results}
\setcounter{equation}{0}

The purpose of this paper is to study the best approximation by polynomials of degree at most $n$ on the 
unit ball $\BB^d: = \{x: \|x\|\le 1\}$ in $\RR^d$. For 
$$
\varpi_\mu(x) = (1-\|x\|^2)^\mu, \qquad \mu > -1,\quad x \in \BB^d,
$$ 
we let $\|\cdot\|_\mu$ be the norm of $L^2(\varpi_\mu;\BB^d)$, defined by 
$$
  \|f \|_\mu := \left(b_\mu \int_{\BB^d} |f(x)|^2 \varpi_\mu(x) dx\right)^{1/2}, 
$$
where $b_\mu = 1/\int_{\BB^d} \varpi_\mu(x) dx$. Let $\Pi_n^d$ the space of polynomials of degree at
most $n$ in $d$ variables. We consider the error, $E_n(f)_\mu$, of best approximation by polynomials 
in $\Pi_n^d$ in the space $L^2(\varpi_\mu; \BB^d)$, defined by 
$$
    E_n(f)_\mu: = \inf_{p_n \in \Pi_n^d} \|f - p_n\|_\mu. 
$$
For $d =1$, the study of this quantity is classical and in the core of approximation theory. For $d > 1$, this quantity 
was characterized in \cite{X05} by a modulus of smoothness defined through a weighted translation 
operator and, more recently, in \cite{DX10} by another modulus of smoothness that is more intuitive and easier to 
compute. Approximation by polynomials on the unit ball has been studied and used recently in the spectral 
method for numerical solutions of partial differential equations \cite{ACH1, ACH2, ACH3, LX, STW}, where 
approximation in $L^2$ norm is most relevant.  

However, many problems remain open. For $d =1$,  a useful estimate for differentiable functions is
$E_n(f)_\mu \le c n^{-1} E_{n-1}(f')_{\mu+1}$, which can be easily verified (see, for example, \cite{X16}). 
Our goal in this paper is to extend this estimate to the unit ball $\BB^d$ with $d > 1$. The problem is more subtle 
than it looks in the first sight. For example, the obvious extension $E_n(f)_\mu \le c n^{-1} 
\sum_{i=1}^d E_{n-1}(\partial_i f)_{\mu+1}$, where $\partial_i$ is the $i$th partial derivative, does not appear to hold on the unit
ball, as our proof indicates (see the discussion below and Remark \ref{rem:final}).

For $\mb \in \NN_0^d$, let 
$|\mb| = m_1 + \ldots + m_d$ and $\partial^\mb := \partial_1^{m_1} \cdots \partial_d^{m_d}$. 
For $s \in \NN$, we denote by $W_2^s(\varpi_\mu, \BB^d)$ the Sobolev space 
$$
W_2^s(\varpi_\mu, \BB^d):=\{f \in L^2(\varpi_\mu, \BB^d):\partial^{\mb} f \in L^2(\varpi_{\mu+|\mb|}, \BB^d),\, |\mb| \leq s, \,
 \mb \in \NN_0^d\}.
$$
Let $\Delta$ denote the usual Laplace operator $\Delta = \partial_1^2+ \cdots + \partial_d^2$ and $\Delta_0$ denotes
the Laplace-Beltrami operator on the unit sphere $\sph$ of $\RR^d$. In spherical-polar coordinates $x = r \xi$, $r \ge 0$ 
and $\xi \in \sph$,  
\begin{equation}\label{L-B-operator}
  \Delta = \frac{d^2}{dr^2} + \frac{d-1}{r} \frac{d}{dr} + \frac{1}{r^2} \Delta_0.
\end{equation}
We will also need another family of operators, $D_{i,j}$, defined by 
$$
    D_{i,j} : = x_i \partial_j - x_j \partial_i, \qquad 1 \le i < j \le d.
$$
These are angular derivatives since $D_{i,j} = \partial_{\t_{i,j}}$ in the polar coordinates of the $(x_i,x_j)$ plane,
$(x_i,x_j) = r_{i,j} (\cos \t_{i,j}, \sin \t_{i,j})$. Furthermore, the angular derivatives $D_{i,j}$ and the Laplace-Beltrami 
operator $\Delta_0$ are related by \cite[p. 24]{DaiX} 
\begin{equation}\label{eq:Delta0=}
   \Delta_0 = \sum_{1\le i< j\le d} D_{i,j}^2.  
\end{equation}

Our main results are the following two theorems: 


\begin{thm} \label{thm:1.1}
Let $s \in \NN$ and let $f \in W_2^{2s}(\varpi_\mu, \BB^d)$. Then, for $n \ge 2s$,
\begin{equation} \label{main-Delta}
E_n(f)_{\mu} \leq \frac{c}{n^{2s}}\left[E_{n-2s}(\Delta^s f)_{\mu+2s} + E_n(\Delta_0^s f)_{\mu}\right].
\end{equation} 
\end{thm}
 
Throughout this paper, the constant $c$ is independent of $n$ and $f$, but may depend on $\mu$, $d$ and $s$; its value 
may be different at different occurrences. 
 
It is easy to see that both terms in the right hand side of \eqref{main-Delta} are necessary. Indeed, if $f$ is a harmonic
function, then $\Delta f = 0$ and we need $E_n(\Delta_0^s f)_{\mu}$, whereas if $f$ is a radial function, $f(x) = f_0(\|x\|)$,
then $\Delta_0 f =0$ and we need $E_{n-2r}(\Delta^s f)_{\mu+2s}$. 
 
\begin{thm} \label{thm:1.3}
Let $s \in \NN_0$ and let $f \in W_2^{2s+1}(\varpi_\mu, \BB^d)$. Then, for $n \ge 2s+1$, 
\begin{align}\label{eq:thm1.3}
  E_n(f)_\mu \le \frac{c}{n^{2s+1}} &  \left[ \sum_{i=1}^d E_{n-2s-1}(\partial_i\Delta^s f)_{\mu+2s+1} 
      + \sum_{1\le i < j \le d} E_n(D_{i,j} \Delta_0^s f)_\mu \right].
\end{align}
\end{thm}

In particular, for $s =0$, the estimate \eqref{eq:thm1.3} states that 
\begin{equation} \label{main-first}
E_n(f)_{\mu} \leq \frac{c}{n}\left[ \sum_{i=1}^d E_{n-1}(\partial_i f)_{\mu+1} + \sum_{1\le i<j \le d} E_n(D_{i,j} f)_{\mu}\right].
\end{equation}
Our proof indicates that the second sum in the estimate \eqref{eq:thm1.3}, the spherical derivatives, is necessary; see 
Remark \ref{rem:final}. Another indication that it is necessary can be seen from the study in \cite{DX10}, where the 
characterization of $E_n(f)_\mu$ is given in terms of a modulus of smoothness, whose equivalent $K$-functional is defined by
$$
K_s(f;t)_\mu = \inf_{g} \left \{\|f-g\|_\mu  + t^s \max_{1 \le i \le d} \|\partial_i^s g\|_{\mu+s} + t^s \max_{i< j} \|D_{i,j}^s g\|_\mu \right\}.
$$

We note that \eqref{main-Delta} does not follow from iteration of \eqref{main-first}. Moreover, \eqref{eq:thm1.3} does not 
follow directly from the combination of \eqref{main-Delta} and \eqref{main-first}.
 
By its definition, $E_n(f)_\mu \le \|f\|_\mu$, which allows us to state the estimates in the right hand side
of \eqref{main-Delta} and \eqref{eq:thm1.3} in terms of norms of the derivatives. In particular, we have 
the following corollary. 

\begin{cor} 
Let $s \in \NN$ and let $f \in {W}_2^{2s}(\varpi_\mu, \BB^d)$. Then, for $n \ge 2s$,
\begin{equation} \label{main-Delta-cor}
E_n(f)_{\mu} \leq \frac{c}{n^{2s}}\left(\|\Delta^s f\|_{\mu+2s} + \|\Delta_0^s f\|_{\mu}\right).
\end{equation} 
\end{cor}

The proof of these results are based on the Fourier expansions in orthogonal polynomials with respect to $\varpi_\mu$. 
The key ingredients are the commuting relations between partial derivatives and the orthogonal projection operators,
and explicit formulas for an explicit basis of orthogonal polynomials and their derivatives. The latter distinguishes this 
study from most of the previous works in this direction, which rely on the closed formula of the reproducing kernel 
instead of working with an explicit basis of orthogonal polynomials. The relations between the orthogonal polynomials 
and their derivatives depend on corresponding relations for an explicit basis of spherical harmonics, which are of 
independent interest.  

The paper is organized as follows. In the next section we recall background materials on orthogonal polynomials and 
orthogonal expansions, and derive several properties on the derivatives of an explicit basis of orthogonal polynomials,
which relies on recursive relations for the derivatives of an explicit basis of spherical harmonics. The recursive 
relations can be used to derive explicit formulas that express the partial derivatives of a family of spherical harmonics 
as a sum of spherical harmonics of one degree lower. The latter relations are of independent interest and we state
explicit formulas for the lower dimensional cases in the Appendix. The Fourier orthogonal expansions on the unit ball 
will be studied in Section 3, which includes the proof of our main theorems.

\section{Spherical harmonics and orthogonal polynomials on the unit ball}
\setcounter{equation}{0}

For $\varpi_\mu(x) = (1-\|x\|^2)^\mu$, $\mu > -1$, on the unit ball $\BB^d$, we define an inner product
\begin{equation} \label{eq:ordinary-ip}
  \la f, g\ra_\mu : = b_\mu \int_{\BB^d} f(x) g(x) \varpi_\mu(x) dx,
\end{equation}
where $b_\mu$ is a the normalizing constant so that $\la 1 ,1 \ra_\mu = 1$. A polynomial $P \in \Pi_n^d$ is called
an orthogonal polynomial of degree $n$ if $\la P, Q\ra_\mu =0$ for all $Q \in \Pi_{n-1}^d$. For $n \in \NN_0$, let $\CV_n^d(\varpi_\mu)$
be the space of orthogonal polynomials of degree $n$ with respect to $\la \cdot,\cdot\ra_\mu$. A mutually orthogonal 
basis of $\CV_n^d(\varpi_\mu)$ can be given in terms of the Jacobi polynomials and spherical harmonics. We will need 
a number of relations on this basis, some of which are new and will be derived in this section. 

In the first subsection below, we recall a basis of spherical harmonics and derive several new properties. The mutually 
orthogonal polynomial basis for $\CV_n^d(\varpi_\mu)$ will be studied in the second subsection. 

\subsection{Spherical harmonics}
Let $\mathcal{H}_n^d$ denote the space of harmonic polynomials of degree $n$. It is known that
$$
    \dim \mathcal{H}_n^d = \binom{n+d-1}{n} - \binom{n+d-3}{n} := a_n^d.
$$
The restriction of $Y \in \mathcal{H}_n^d$ on $\SS^{d-1}$ are called spherical harmonics, which 
are eigenfunctions of the Laplace-Beltrami operator; more specifically, 
\begin{equation}\label{L-B-eigen}
\Delta_0 Y (\xi) = -n(n+d-2)Y(\xi), \quad \forall \,Y\, \in\,\mathcal{H}_n^d, \quad \xi \in \SS^{d-1}.
\end{equation}
It is known that $\Delta_0$ is self-adjoint and, moreover, by \eqref{eq:Delta0=} and the self-adjointness of 
$D_{i,j}$ \cite[(1.8.7)]{DaiX}, 
\begin{equation} \label{eq:intDelta0}
  \int_{\sph} \Delta_0 f(x) g(x) d\s(x) = \int_{\sph}\sum_{1 \le i < j \le d} D_{i,j} f(x) D_{i,j} g(x) d\sigma(x). 
\end{equation}
Spherical harmonics are orthogonal polynomials on $S^{d-1}$. A basis $\{Y_\nu^n: 1\le \nu \le a_{n}^d\}$ is 
a mutually orthogonal basis for $\mathcal{H}_{n}^d$ if 
\begin{equation}\label{sphhar-ort}
\frac{1}{\s_{d-1}} \, \int_{S^{d-1}} \, Y^n_\nu(\xi)\, Y^m_\eta(\xi)\, d\s(\xi) =  h_\nu^n \delta_{n,m}\,\delta_{\nu,\eta},
    \quad 1\le \nu \le a_n^d, \quad 1\le \eta \le a_m^d,
\end{equation}
where $d\s$ is the surface measure on $S^{d-1}$, $\s_{d-1}$ is the surface area of $S^{d-1}$ and $h_\nu^n$ is 
the $L^2$ norm of $Y_\nu^n$. The basis is called orthonormal if $h_\nu^n =1$. If $\{Y_\nu^n\}$ is an orthonormal
basis, then it satisfies the addition formula 
$$
  \sum_{1 \le \nu \le a_n^d} Y_\nu^n(\xi) Y_\nu^n(\varrho) = \frac{n+\l}{\l} C_n^\l (\la \xi,\varrho \ra), \qquad \l = \f{d-2}{2},\quad 
   \xi,\varrho \in \sph, 
$$
where $C_n^\l$ is the $n$th Gegenbauer, or ultraspherical, polynomial. 

For our study, we need to work with an explicit mutually orthogonal basis for $\CH_n^d$, which we choose to be 
the standard one given in terms of the Gegenbauer polynomials (see, for example, \cite{DaiX, DX}). However,
instead of writing this basis in spherical coordinates, we choose to write them in cartesian coordinates. The basis 
below is given in \cite[p. 115]{DX} but with the order of the variables inverted for our convenience. 

Let $T_n(t)$ and $U_n(t)$ denote the Chebyshev polynomials of the first and the second kind, respectively. 
Define 
\begin{align} \label{eq:harm-d=2}
\begin{split}
g_{0,n} (x_1, x_2) &= (x_1^2 + x_2^2)^{n/2} 
T_{n}\left( x_2 (x_1^2 + x_2^2)^{-1/2}\right),\\
g_{1,n-1} (x_1, x_2) &= x_1 (x_1^2 + x_2^2)^{(n-1)/2} 
U_{n-1}\left( x_2 (x_1^2 + x_2^2)^{-1/2}\right).
\end{split}
\end{align}
Evidently these are homogeneous polynomials of degree $n$ and $\{g_{0,n}, g_{1,n-1}\}$ constitutes a mutually 
orthogonal basis for $\mathcal{H}_n^2$. For $d>2$ and $\mathbf{n} = (n_1, n_2, n_3, \ldots , n_d) \in \NN_0^d$ 
with $n_1 = 0$ or $1$, define 
\begin{equation} \label{spher-basis}
Y_{\mathbf{n}}(x) = g_{n_1,n_2}(x_1,x_2) \prod_{j=3}^d (x_1^2 + \ldots + x_j^2)^{n_j/2} 
C_{n_j}^{\l_j} \left( x_j (x_1^2 + \ldots + x_j^2)^{-1/2}\right),
\end{equation}
where 
$$
\l_j = \l_j (n_1,\ldots, n_{j-1}) := \sum_{i=1}^{j-1} n_i + \frac{j-2}{2}.
$$
Then $\{Y_\mathbf{n}; |\mathbf{n}| = n \,\, \hbox{with} \,\, n_1 = \hbox{$0$ or $1$}\}$ is a mutually orthogonal basis 
of $\mathcal{H}_n^d$. We need information on two operations on this basis, one is partial derivatives $\partial_i$ 
and the other is multiplication by $x_i$. They are related by the orthogonal projection operator 
$$
\proj_{n,\SS}^d: \CP_n^d \mapsto \CH_n^d
$$
from $\CP_n^d$, the space of homogeneous polynomials of degree $n$ in $d$ variables, into $\CH_n^d$. Indeed,
it is known (cf. \cite[(1.2.1)]{DaiX}) that 
$$
   \proj_{n,\SS}^d P = \sum_{j=0}^{\lfloor \f{n}2 \rfloor} \frac{1}{4^j j! (-n+2-d/2)_j} \|x\|^{2j}\Delta^j P, 
$$
which implies, for a spherical harmonic $Y_{\mathbf{n}} \in \CH_n^d$, 
\begin{equation} \label{proj-def}
   \proj_{n+1,\SS}^d (x_i Y_{\mathbf{n}}(x)) = x_i Y_{\mathbf{n}}(x) 
           - \frac{1}{2 (n+(d-2)/2)} \|x\|^{2} \partial_i Y_{\mathbf{n}}(x). 
\end{equation}
The basis \eqref{spher-basis} satisfies the following property, which is of independent interest. 

\begin{thm} \label{thm:diffY}
Let $\mathbf{n} = (n_1, n_2, \ldots, n_d) \in \NN_0^d$ with $n_1 = 0$ or $1$ and $|\mathbf{n}| = n$. Then 
\begin{enumerate}
\item $\partial_i Y_{\mathbf{n}}(x)$ is an spherical harmonic of degree $n-1$ and
$$
\langle  \partial_i Y_{\mathbf{n}}, Y_{\mathbf{m}} \rangle_{\mathbb{S}^{d-1}} \neq 0, \qquad |\mathbf{m}| =n-1
$$
for at most $2^{d-2}$ many $\mathbf{m} \in \NN_0^d$ with $m_1 = 0$ or $1$. 
\item Let  $x_i$ denote also the operator of multiplication by $x_i$. Then 
$$
\langle \proj_{n+1,\SS}^{d} (x_i Y_{\mathbf{n}}), Y_{\mathbf{m}} \rangle_{\mathbb{S}^{d-1}} \neq 0, \qquad |\mathbf{m}| = n+1,
$$
for at most $2^{d-2}$ many $\mathbf{m} \in \NN_0^d$ with $m_1 =0$ or $1$. 
\end{enumerate}
\end{thm}

Our proof below will give recursive formulas that can be used to derive explicit expressions for writing 
$\partial_i Y_\mathbf{n}$ as a linear combination of $Y_\mathbf{m}$. For lower dimensions, these formulas are given
in the Appendix. For large $d$, the formulas can be complicated. For our purpose, however, we do not need explicit 
formulas. 

The key ingredient of the recursive relations for the proof of the theorem lies in the product structure of $Y_\mathbf{n}$. 
For $n\in \mathbb{N}$ and $\lambda > -1/2$, define
\begin{align*}
F_n^{\l} (x) : = (x_1^2 + \ldots + x_d^2)^\frac{n}{2} 
C_{n}^{\l} \left( \frac{x_d}{\sqrt{x_1^2 + \ldots + x_d^2}}\right).
\end{align*}
Let $d>2$. For $x = (x_1, x_2, \ldots, x_d) \in \RR^d$ and $\mathbf{n} = (n_1, n_2, \ldots, n_d) \in \NN^d$, we 
denote $x' = (x_1, x_2, \ldots, x_{d-1})$ and $\mathbf{n'} = (n_1, n_2, \ldots, n_{d-1})$. Then, by \eqref{spher-basis},
\begin{equation} \label{eq:Yn-recur}
Y_\mathbf{n}(x) = Y_\mathbf{n'}(x') F_{n_d}^{\l_d}(x),
\end{equation}
where $Y_\mathbf{n'}(x')$ is a spherical harmonic in $\CH_{n-n_d}^{d-1}$. 

\begin{lem} 
The partial derivatives of $F_n^\l$ satisfy
\begin{align} \label{partial_i_F}
\begin{split}
\partial_i F_n^{\l} (x) &= - 2\l x_i F_{n-2}^{\l+1} (x), \qquad i = 1, 2, \ldots, d-1, \\
\partial_d F_n^{\l} (x) & = (n + 2\l - 1) F_{n-1}^{\l} (x).  
\end{split}
\end{align}
\end{lem}

\begin{proof}
Let $r = \sqrt{x_1^2 + \ldots + x_d^2}$. Using the fact that $\f{d}{dx} C_n^\l(x) = 2 \l C_{n-1}^{\l + 1}(x)$, 
we obtain, for $i = 1, 2, \ldots, d-1$, 
\begin{align*}
\partial_i F_n^{\l} (x)  &= x_i r^{n-2} \left[ n C_{n}^{\l} \left( \frac{x_d}{r}\right) 
- 2 \l \frac{x_d}{r} C_{n-1}^{\l+1} \left( \frac{x_d}{r}\right)\right]\\
&= - 2 \l x_i r^{n-2} C_{n-2}^{\l+1}\left(\frac{x_d}{r}\right),
\end{align*}
where the second step follows from \cite[(4.7.28)]{Sz}. Moreover, for $i=d$, we obtain
\begin{align*}
\partial_d F_n^{\l} (x)  &= r^{n-1} \left[ n \frac{x_d}{r} C_{n}^{\l} \left( \frac{x_d}{r}\right) 
    + 2 \l \left(1 - \frac{x_d^2}{r^2}\right) C_{n-1}^{\l+1} \left( \frac{x_d}{r}\right)\right]\\
&= (n + 2 \l - 1) r^{n-1} C_{n-1}^{\l}\left(\frac{x_d}{r}\right),
\end{align*}
where the second step follows from the second equality of \cite[(4.7.27)]{Sz}. 
\end{proof}
 
\begin{prop} \label{prop_partial}
Let $n' := |\mathbf{n}'| = n - n_d$. For $i = 1, 2, \ldots, d-1$, 
\begin{align} 
\partial_i Y_{\mathbf{n}}(x) =& - 2\l_d \proj_{n'+1,\SS}^{d-1} (x_i Y_{\mathbf{n'}}(x')) 
F_{n_d-2}^{\l_d+1}(x)    \label{partial_i_Y} \\
 &+ \frac{(n_d + 2\l_d -1)(n_d + 2\l_d -2)}{(2\l_d -1)(2\l_d -2)}
 \partial_i Y_{\mathbf{n'}}(x') F_{n_d}^{\l_d-1} (x), \notag \\
\partial_d Y_{\mathbf{n}}(x) =&\, (n_d + 2\l_d - 1) Y_{\mathbf{n'}}(x') F_{n_d-1}^{\l_d} (x). \label{partial_d_Y}
\end{align}
\end{prop}

\begin{proof}
For $i = 1, 2, \ldots, d-1$, using \eqref{partial_i_F} and \eqref{proj-def} we deduce
\begin{align*}
\partial_i Y_{\mathbf{n}}(x) =& \partial_i Y_\mathbf{n'}(x') \, F_{n_d}^{\l_d}(x) 
  - 2 \l_d x_i Y_\mathbf{n'}(x') F_{n_d-2}^{\l_d+1}(x) \\
=& - 2 \l_d \proj_{n'+1,\SS}^{d-1} (x_i Y_{\mathbf{n'}}(x')) F_{n_d-2}^{\l_d+1}(x)\\
&+ \partial_i Y_\mathbf{n'}(x') \left[ F_{n_d}^{\l_d}(x) - \frac{2\l_d}{2\l_d - 1} \|x'\|^2 F_{n_d-2}^{\l_d+1}(x)\right],
\end{align*}
where we have used the fact that $2n' + d-3 = 2 \l_d -1$. Since $\|x'\|^2 = r^2 - x_d^2$, we can write
\begin{align*}
 F_{n_d}^{\l_d}(x) &- \frac{2\l_d}{2\l_d - 1} \|x'\|^2 F_{n_d-2}^{\l_d+1}(x)\\
 &= r^{n_d} \left[ C_{n_d}^{\l_d} \left( \frac{x_d}{r}\right)  - \frac{2\l_d}{2\l_d - 1}
  \left(1 - \frac{x_d^2}{r^2}\right)  C_{n_d-2}^{\l_d+1} \left( \frac{x_d}{r}\right) \right] \\
 &= \frac{(n_d + 2\l_d -1)(n_d + 2\l_d -2)}{(2\l_d -1)(2\l_d -2)} r^{n_d}  C_{n_d}^{\l_d-1} \left( \frac{x_d}{r}\right),
\end{align*}
where the second step follows from the first identity of \cite[(4.7.27)]{Sz}, using first $2 \l_d C_{n_d-2}^{\l_d+1} =
\frac{d}{dx} C_{n_d-1}^{\l_d}$, and the second identity of \cite[(4.7.29)]{Sz}. This establishes \eqref{partial_i_Y}. 
Finally, \eqref{partial_d_Y} is a direct consequence of the second identity in \eqref{partial_i_F}.
\end{proof}

A similar recurrence relation can be deduced for the projection operator $\proj_{n+1,\SS}^{d}$ applied
to $x_i Y_{\mathbf{n}}(x)$.

\begin{prop} \label{prop_proj}
Let $n' = |\mathbf{n}'| = n - n_d$. For $i = 1, 2, \ldots, d-1$, 
\begin{align} 
\proj_{n+1,\SS}^{d} (x_i Y_{\mathbf{n}}(x)) =&\,  \frac{\l_d}{n_d+\l_d} 
\proj_{n'+1,\SS}^{d-1} (x_i Y_{\mathbf{n'}}(x')) F_{n_d}^{\l_d+1}(x) -  \label{proj_i_Y} \\
 &- \frac{(n_d + 1)(n_d + 2)}{(2\l_d -1)(2\l_d -2)2(n_d+\l_d)}
\partial_i Y_{\mathbf{n'}}(x') F_{n_d+2}^{\l_d-1} (x), \notag \\
\proj_{n+1,\SS}^{d} (x_d Y_{\mathbf{n}}(x)) =&\, \frac{n_d+1}{2(n_d + \l_d)} Y_{\mathbf{n'}}(x') F_{n_d+1}^{\l_d} (x). \label{proj_d_Y}
\end{align}
\end{prop}
\begin{proof}
For $i = 1, 2, \ldots, d-1$, using \eqref{proj-def} and \eqref{eq:Yn-recur}, we obtain
\begin{align*}
\proj_{n+1,\SS}^{d} & (x_i Y_{\mathbf{n}}(x))
  = x_i Y_\mathbf{n'}(x') \, F_{n_d}^{\l_d}(x) 
- \frac{r^2}{2(n_d + \l_d)}  \partial_i \left(Y_\mathbf{n'}(x') \, F_{n_d}^{\l_d}(x)\right),
\end{align*}
which implies, by \eqref{partial_i_Y}, that 
\begin{align*}
 \proj_{n+1,\SS}^{d} & (x_i Y_{\mathbf{n}}(x)) = \proj_{n'+1,\SS}^{d-1} (x_i Y_{\mathbf{n'}}(x')) \left[ F_{n_d}^{\l_d}(x) +  
\frac{\l_d}{n_d + \l_d} r^2 F_{n_d-2}^{\l_d+1}(x)\right] \\
&+ \partial_i Y_{\mathbf{n'}}(x') \left[ \frac{\|x'\|^2}{2\l_d - 1} F_{n_d}^{\l_d}(x) 
- \frac{(n_d + 2\l_d -1)(n_d + 2\l_d -2)}{2(n_d+\l_d)(2\l_d - 1)(2\l_d -2)} r^2 F_{n_d}^{\l_d-1}(x)
\right].
\end{align*}
The terms in the first bracket in the right hand side are equal to, by the second identity in \cite[(4.7.29)]{Sz}, 
\begin{align*}
  r^{n_d} \left[C_{n_d}^{\l_d}\left( \frac{x_d}{r}\right) +  \frac{\l_d}{n_d + \l_d} C_{n_d-2}^{\l_d+1}\left( \frac{x_d}{r}\right)\right]
= \frac{\l_d}{n_d + \l_d} r^{n_d} C_{n_d}^{\l_d+1}\left( \frac{x_d}{r}\right), 
\end{align*}
whereas the terms in the second bracket are equal to, using the first identity of \cite[(4.7.27)]{Sz} again, 
\begin{align*}
& \frac{r^{n_d+2}}{2\l_d - 1}\left[ \left(1 - \frac{x_d^2}{r^2}\right) C_{n_d}^{\l_d}\left( \frac{x_d}{r}\right) 
- \frac{(n_d + 2\l_d -1)(n_d + 2\l_d -2)}{2(n_d+\l_d)(2\l_d -2)} C_{n_d}^{\l_d-1}\left( \frac{x_d}{r}\right)\right] \\
& \qquad\qquad\qquad 
 = - \frac{(n_d +1)(n_d + 2)}{2(n_d+\l_d)(2\l_d -1)(2\l_d -2)} r^{n_d+2} C_{n_d+2}^{\l_d-1}\left( \frac{x_d}{r}\right).
\end{align*}
Putting these together, we have established \eqref{proj_i_Y}. The proof of \eqref{proj_d_Y} can be carried out
similarly as follows, 
\begin{align*}
\proj_{n+1,\SS}^{d} (x_d Y_{\mathbf{n}}(x))  
 & = x_d Y_\mathbf{n'}(x') \, F_{n_d}^{\l_d}(x) 
- \frac{r^2}{2(n_d + \l_d)}  \partial_d \left(Y_\mathbf{n'}(x') \, F_{n_d}^{\l_d}(x)\right)\\
&= Y_\mathbf{n'}(x') \left[ x_d F_{n_d}^{\l_d}(x) 
- \frac{r^2}{2(n_d + \l_d)} (n_d + 2\l_d-1)   F_{n_d-1}^{\l_d}(x)\right]\\
&= \frac{n_d+1}{2(n_d + \l_d)} Y_{\mathbf{n'}}(x') F_{n_d+1}^{\l_d} (x),
\end{align*}
where the last step follows from the three term recurrence relation of the Gegenbauer polynomials \cite[(4.7.17)]{Sz}.
\end{proof}

\medskip

\noindent
{\it Proof of Theorem \ref{thm:diffY}}
The fact that $\partial_i Y_{\mathbf{n}}(x)$ is an element of $\CH_{n-1}^d$ follows from the commutativity of 
$\partial_i$ and $\Delta$. By its definition, $\proj_{n+1,\SS}^d (x_i Y_\mathbf{n})$ is an element of $\CH_{n+1}^d$. 
Let $Y_{\mathbf{n}}^n = Y_{\mathbf{n}}$, where the superscript $n = |\mathbf{n}|$ is added to indicate its degree. 
The statement of the theorem is equivalent to 
$$
   \partial_i Y_\mathbf{n} = \sum_{\mb \in \Lambda_1} a_\mb Y_\mb^{n-1} \quad \hbox{and} \quad 
    \proj_{n+1,\SS}^d (x_i Y_\mathbf{n}) = \sum_{\mb \in \Lambda_2} b_\mb Y_\mb^{n+1},
$$
where $a_m$ and $b_m$ are real constants, $\Lambda_1$ and $\Lambda_2$ consist of at most $2^{d-2}$ elements
in $\NN_0^d$ with $m_1 = 0$ or $1$. This equivalent statement can be established by induction on the dimension $d$,
using the equalities in Proposition \ref{prop_partial} and Proposition \ref{prop_proj}. The case $d =2$ can be explicitly 
written down, which is given in the Appendix. Assume that the statement has been proved in the case of $d-1$ variables
for $d \ge 3$. Then $\partial_i Y_\mathbf{n'}^{n'}(x')$ can be written as a linear combination of at most $2^{d-3}$ many 
$Y_{\mb'}^{n'-1}$ and $\proj_{n'+1,\SS}^d (x_i Y_\mathbf{n'})$ can be written as a linear combination of at most $2^{d-3}$ 
many $Y_{\mb'}^{n'+1}$. By \eqref{eq:Yn-recur}, it is easy to see that $Y_{\mb'}^{n'-1}(x') F_{n_d}^{\l_d-1}(x) = Y_{m_1,\ldots,m_{d-1},n_d}^{n-1}$ and $Y_{\mb'}^{n'+1}(x') F_{n_d}^{\l_d+1}(x)  =  Y_{m_1,\ldots,m_{d-1},n_d}^{n+1}$. 
As a consequence, we can use the identities \eqref{partial_i_Y} and \eqref{proj_i_Y} to complete the proof.  
\qed

\subsection{Orthogonal polynomials on the unit ball} 
A family of mutually orthogonal polynomials with respect to the inner product $\la \cdot, \cdot \ra_\mu$, defined in
\eqref{eq:ordinary-ip}, can be given in terms of the Jacobi polynomials and spherical harmonics. 

For $\a, \b > -1$, Jacobi polynomials $P_n^{(\a,\b)}$ are defined by 
\begin{equation*} 
P_n^{(\a,\b)}(t) = \frac{(\a+1)_n}{n!} {}_2F_1 \left(  \begin{matrix} -n, n+\a+\b+1 \\ \a+1 \end{matrix};  \frac{1-x}{2} \right),
\end{equation*}
where $(a)_n= a(a+1) \ldots (a+n-1)$ denotes the Pochhammer symbol. They are orthogonal with respect to the 
Jacobi weight function $w_{\a,\b}(t) := (1-t)^\alpha(1+t)^{\beta}$ on $[-1,1]$.

For $n \in \NN_0$ and $0 \le j \le n/2$, let $\{Y_\nu^{n-2j}: 1\le \nu\le a_{n-2j}^d\}$ denote
an orthonormal basis for $\mathcal{H}_{n-2j}^d$. For $\mu > -1$, define \cite[p.142]{DX}
\begin{equation}\label{baseP}
P_{j,\nu}^{n,\mu}(x):= P_{j}^{(\mu, n-2j + \frac{d-2}{2})}(2\,\|x\|^2 -1)\, Y_\nu^{n-2j}(x). 
\end{equation}
Then the set $\{P_{j,\nu}^{n,\mu}: 0 \le j \le n/2, \,1 \le \nu \le a_{n-2j}^d \}$ consists of a mutually
orthogonal basis of $\CV_n^d(\varpi_\mu)$; more precisely, 
$$
\la P_{j,\nu}^{n,\mu}, P_{k,\eta}^{m,\mu}\ra_\mu =  h_{j,n}^{\mu}  \delta_{n,m}\,\delta_{j,k}\,\delta_{\nu,\eta},
$$
where $h_{j,n}^{\mu}$ is given by 
\begin{equation} \label{eq:Hjn-mu}
 h_{j,n}^{\mu}: = \frac{(\mu +1)_j (\frac{d}{2})_{n-j} (n-j+\mu+ \frac{d}{2})}
    { j! (\mu+\frac{d+2}{2})_{n-j} (n+\mu+ \frac{d}{2})}.
\end{equation} 
The orthogonal basis $\{P_{j,\nu}^{n,\mu}\}$ satisfies two other orthogonal relations in the Sobolev space. 

\begin{lem} 
Let $\mu >-1$. Then the basis $\{P_{j,\nu}^{n,\mu}\}$ satisfies 
\begin{equation} \label{eq:nabla-P}
  b_{\mu} \int_{\BB^d} \nabla P_{j,\nu}^{n,\mu} (x)\cdot \nabla P_{j',\nu'}^{m,\mu} (x) \varpi_{\mu+1}(x) dx = h_{j,n}^\mu(\nabla)
    \delta_{\nu,\nu'} \delta_{j,j'} \delta_{n,m},
\end{equation}
where 
$$
   h_{j,n}^\mu(\nabla) = (4j(n-j+\mu+d/2)+2(n-2j)(\mu+1)) h_{j,n}^\mu.
$$
\end{lem}

\begin{proof}
The orthogonality \eqref{eq:nabla-P} was stated in \cite{PPX}, but the norm $h_{j,n}^\mu(\nabla)$ there is incorrect. 
Since $\varpi_{\mu+1}(x)$ vanishes on the sphere, Green's identity implies
\begin{align*}
\int_{\BB^d}& \nabla P_{j,\nu}^{n,\mu} (x) \cdot \nabla P_{k,\eta}^{m,\mu} (x) 
\varpi_{\mu+1}(x) dx  \\
&= - \int_{\BB^d} P_{k,\eta}^{m,\mu} (x) \left[(1-\|x\|^2) \Delta 
- 2(\mu+1) \la x, \nabla \ra \right] P_{j,\nu}^{n,\mu} (x) \varpi_{\mu}(x) dx.
\end{align*}
Let $\b_j := n-2j + \frac{d-2}{2}$. Working with spherical--polar coordinates, by \eqref{L-B-operator} and
$\la x, \nabla \ra = r \frac{d}{dr}$, a straightforward computation shows that 
\begin{align*}
  \big[(1& -\|x\|^2) \Delta    
- 2(\mu+1) \la x, \nabla \ra \big] P_{j,\nu}^{n,\mu} (x) \\
 = & \left[16 (r^2-r^4) (P_j^{(\mu, \b_j)})''(2r^2 - 1) -8 ((\mu+2+\beta_j)r^2 - \beta_j -1))(P_j^{(\mu, \b_j)})'(2r^2 - 1) \right. \\
      & \left.- 2 (\mu+1)(n-2j)P_j^{(\mu, \b_j)}(2r^2 - 1) \right] r^{n-2j} Y_\nu^{n-2j}(\xi) \\
  = & - \left[4j(j+\mu+\beta_j+1) + 2 (\mu+1)(n-2j)\right] P_{j,\nu}^{n,\mu} (x),
\end{align*}
where the second identity follows from the second order differential equation \cite[(4.2.1)]{Sz} satisfied by the 
Jacobi polynomial $P_j^{(\mu,\beta_j)}$. Consequently, 
\begin{align*}
&b_\mu \int_{\BB^d} \nabla P_{j,\nu}^{n,\mu} (x) \cdot \nabla P_{k,\eta}^{m,\mu} (x) 
\varpi_{\mu+1}(x) dx =\\ 
& \qquad = \lambda_{j,n}^\mu b_\mu \int_{\BB^d} P_{j,\nu}^{n,\mu} (x) P_{k,\eta}^{m,\mu} (x) 
\varpi_{\mu}(x) dx 
= \lambda_{j,n}^\mu h_{j,n}^\mu \delta_{n,m} \delta_{k,j} \delta_{\nu,\eta},
\end{align*}
where $h_{j,n}^\mu$ is defined in \eqref{eq:Hjn-mu},  and
$\lambda_{j,n}^\mu = 4j(j+\mu+\beta_j+1) + 2 (\mu+1)(n-2j)$. 
\end{proof}

\begin{lem} \label{lem:2.2}
Let $\mu >-1$. Then the basis $\{P_{j,\nu}^{n,\mu}\}$ satisfies 
\begin{equation} \label{eq:Dij-P}
  b_\mu \int_{\BB^d} \sum_{1 \le i< j \le d} D_{i,j} P_{\ell,\nu}^{n,\mu} (x) D_{i,j} P_{\ell',\nu'}^{m,\mu} (x) \varpi_{\mu}(x) dx 
    = h_{\ell,n}^\mu(D) \delta_{\nu,\nu'} \delta_{\ell,\ell'} \delta_{n,m},
\end{equation}
where 
$$
   h_{\ell,n}^\mu(D) = (m-2\ell) (m-2 \ell +d-2) h_{\ell,n}^\mu.
$$
\end{lem}

\begin{proof}
From \eqref{eq:intDelta0} and \eqref{eq:Delta0=}, we obtain immediately that 
\begin{align*}
\int_{\sph} \sum_{i< j} D_{i,j} Y_\nu^{n}(x) D_{i,j} Y_\eta^{m} (x) d\s(x) = m(m+d-2) \int_{\sph} Y_\nu^{n}(x) Y_\eta^{m} (x) d\s(x),  
\end{align*}
form which the stated result follows immediately by writing the integrals in spherical--polar coordinates. 
\end{proof}

For convenience, we define $P_{j,\nu}^{n,\mu} =0$ if $j < 0$. The polynomial $P_{j,\nu}^{n,\mu}$ enjoys a simple
form under both $\Delta$ and $\Delta_0$. 

\begin{lem}\label{lemma_laplace_P}
Let $\mu > -1$ and let $P_{j,\nu}^{n,\mu}$ be defined in \eqref{baseP}. Then 
\begin{equation} \label{laplace_P}
\Delta P_{j,\nu}^{n,\mu}(x) = \k_{n-j}^\mu P_{j-1,\nu}^{n-2,\mu+2}(x) \quad\hbox{and}\quad
   \Delta_0  P_{j,\nu}^{n,\mu}(x) = \l_{n-2j} P_{j,\nu}^{n,\mu}(x),
\end{equation}
where 
$$
\k_{n}^\mu: = 4 (n + \mu + \tfrac{d}{2})(n+\tfrac{d-2}{2})  \quad\hbox{and}\quad \l_n:= -n(n+ d-2).
$$
\end{lem}

\begin{proof}
Let again $\beta_j= n- 2j + \frac{d-2}{2}$. Using \eqref{L-B-operator} in spherical-polar coordinates $x= r \xi$, $\xi \in \sph$,  
it is easy to see that 
\begin{align*}
\Delta P_{j,\nu}^{n,\mu}(x) & = 8 \left(2 r^2 (P_{j}^{(\mu,\beta_j)})''(2r^2 - 1) + 
(\beta_j + 1)(P_{j}^{(\mu,\beta_j)})'(2r^2 - 1)\right)  r^{n-2j} Y_{\nu}^{n-2j}(\xi) \\
& = \left(2 r^2 (P_{j-1}^{(\mu+1,\beta_j+1)})'(2r^2 - 1) + 
(\beta_j + 1)P_{j-1}^{(\mu+1,\beta_j+1)}(2r^2 - 1)\right)\\
& \qquad\qquad\qquad\qquad\qquad \times 4 (j + \beta_j + \mu +  1) r^{n-2j} Y_{\nu}^{n-2j}(\xi) \\
& = 4 (j + \beta_j + \mu +  1) (j + \beta_j)P_{j-1}^{(\mu+2,\beta_j)}(2r^2 - 1) r^{n-2j} Y_{\nu}^{n-2j}(\xi), 
\end{align*}
where the last two equal signs follow, respectively, from the second order differential equation satisfied by the 
Jacobi polynomials and by the identity (\cite[(2.21)]{PPX})
\begin{align}\label{eq:Jac-r2}
\beta P_{j}^{(\alpha,\beta)}(t) + (1+t)\,\frac{d}{d t} P_j^{(\alpha,\beta)}(t) = (\beta + j) P_j^{(\alpha+1,\beta-1)}(t). 
\end{align}
Taking into account that $n-2j = (n-2) - 2(j-1)$, this proves the identity in \eqref{laplace_P} for $j > 0$. For $j=0$, 
$P_{0,\nu}^{n,\mu} = Y_{\nu}^n$, so that $\Delta P_{0,\nu}^{\mu,n}(x) = 0$.

The second identity in \eqref{laplace_P} is a direct consequence of the identity \eqref{L-B-eigen}.
\end{proof}
  
\begin{lem}
For $1 \le i \le d$, let $\wh Y_{\eta, i}^{m+1}: = \proj_{m+1, \SS}^d (x_i Y_\eta^m)$. Let $\beta_\ell = m-2\ell + \frac{d-2}{2}$. 
Then,  
\begin{align} \label{eq:diff-Pbasis}
\partial_i P_{\ell,\eta}^{m,\mu}(x) = & \,
  \frac{\b_\ell + \ell}{\b_\ell} P_\ell^{(\mu+1,\b_\ell-1)}(2r^2-1) \partial_i Y_\eta^{m-2\ell}(x) \\
  & + 2(\ell+\mu+\b_\ell+1) P_{\ell-1}^{(\mu+1,\b_\ell+1)}(2r^2-1)  \wh Y_{\eta,i}^{m-2\ell+1}(x). \notag
\end{align}
\end{lem}

\begin{proof}
A straightforward computation shows that
\begin{align*}
 \partial_i P_{\ell,\eta}^{m,\mu}(x) = &  4 x_i (P_\ell^{(\mu,\b_\ell)})'(2 r^2 -1) Y_\eta^{m-2\ell}(x)  
 + P_\ell^{(\mu,\b_\ell)}(2 r^2 -1) \partial_i Y_\eta^{m-2\ell}(x). 
\end{align*}
Furthermore, by \eqref{proj-def}, 
$$
   \wh Y_{\eta,i}^{m-2\ell+1} (x) = \proj_{m-2\ell+1, \SS}^d (x_i Y_\eta^{m-2\ell}(x)) 
        = x_i Y_\eta^{m-2\ell}(x) - \frac{r^2}{2\beta_\ell} \partial_i Y_\eta^{m-2\ell}(x).
$$
Combining these two identities, we obtain
\begin{align*}
 \partial_i P_{\ell,\eta}^{m,\mu}(x) & = 4 (P_\ell^{(\mu,\b_\ell)})'(2 r^2 -1)  \wh Y_{\eta,i}^{m-2\ell+1}(x) \\
     & + \frac{1}{\b_\ell} \left[2r^2 (P_\ell^{(\mu,\b_\ell)})'(2 r^2 -1) + \b_\ell P_\ell^{(\mu,\b_\ell)}(2 r^2 -1)\right] \partial_i Y_\eta^{m-2\ell}(x). 
\end{align*}
Using the derivative formula of the Jacobi polynomials and \eqref{eq:Jac-r2} we can then verify the stated identity.
\end{proof}

Up to this point, we assumed that the spherical harmonics $Y_\nu^{n-2j}$ in the basis $P_{j,\nu}^{n,\mu}$, defined in \eqref{baseP}, form an orthonormal basis of $\CH_{n-2j}^d$ but did not specify this basis. In our next proposition, however,
we need to specify this basis as the one defined in \eqref{spher-basis}. An orthonormal basis of $\CH_{n-2j}^d$ can be 
derived from the basis in \eqref{spher-basis} up to a normalization. For convenience, we again denote this orthonormal 
basis by $\{Y_\nu^{n-2j}\}$, with the understanding that $\nu \in \NN_0^d$ with $\nu_1 =0$ or $1$ and $|\nu| = n-2j$ now. 

\begin{prop} \label{prop:ballOP}
Let $P_{\ell,\nu}^{n,\mu}$ be defined in \eqref{baseP} with $Y_\nu^{n-2j}$ being the orthonormal basis defined 
in \eqref{spher-basis}. Let $\eta \in \NN_0^d$ with $|\eta| = n-2k$ and $\eta_1 =0$ or $1$. Then
\begin{enumerate}
\item for $1 \le i \le d$, $\la  \partial_i P_{\ell,\nu}^{n,\mu}, P_{k,\eta}^{n-1,\mu+1} \ra_{\mu+1} \ne 0$ only if 
$k = \ell$ or $\ell -1$ and, in each case, for at most $2^{d-1}$ many $\nu \in \NN_0^d$ with $\nu_1 = 0$ or $1$;
\item for $1 \le i< j \le d$, $\la  D_{i,j} P_{\ell,\nu}^{n,\mu}, P_{k,\eta}^{n,\mu} \ra_{\mu+1} \ne 0$ only if $k =\ell$ 
and for at most $2^{2d-1}$ many $\nu \in \NN_0^d$ with $\nu_1 = 0$ or $1$.
\end{enumerate}
\end{prop}

\begin{proof}
Since $\partial_i Y_\nu^{n-2\ell} \in  \CH_{n-2\ell -1}^d$ and $\wh Y_{\nu,i}^{n-2\ell+1} \in \CH_{n-2\ell +1}^d$, 
it is easy to see, by \eqref{eq:diff-Pbasis}, that $P_{\ell,\nu}^{n,\mu}$ can be written as a sum of two polynomials,
both in $\CV_{n-1}^d(\varpi_{\mu+1})$, with one written as a sum of $P_{\ell,\tau}^{n-1,\mu+1}$ over $\tau$ 
and another written as a sum of $P_{\ell-1,\tau}^{n-1,\mu+1}$ over $\tau$; moreover, by Theorem \ref{thm:diffY},
both sums consist of at most $2^{d-2}$ terms, which proves the item (1). 

It is known that $D_{i,j}$ is an angular derivative that applies only on the spherical part and maps $\CH_n^d$ 
into itself \cite[Lemma 1.8.3]{DaiX}. Thus, for $P_{\ell,\nu}^{m,\mu}$, it acts only on $Y_{\nu}^{m-2\ell}$. Since 
$\proj_{m,\SS}^d$ is a linear operator, it follows that 
$$
D_{i,j} Y_\nu^{m-2\ell} = \proj_{m,\SS}^d (D_{i,j} Y_\nu^{m-2\ell}) 
    =  \proj_{m,\SS}^d (x_i \partial_j Y_\nu^{m-2\ell}) -  \proj_{m,\SS}^d (x_j \partial_i Y_\nu^{m-2\ell}).
$$
Consequently, the item (2) follows directly from Theorem \ref{thm:diffY}. 
\end{proof}

Using the explicit formulas of the spherical harmonics $Y_\mb$ and recursive relations in Propositions \ref{prop_partial}
and \ref{prop_proj}, we can use \eqref{eq:diff-Pbasis} to derive explicit formulas for writing 
$\partial_i P_{\ell,\eta}^{n,\mu}$ as a sum of $P_{j,\nu}^{n-1,\mu+1}$.

\section{Fourier orthogonal expansions and approximation}
\setcounter{equation}{0}
 
With respect to the basis \eqref{baseP}, the Fourier orthogonal expansion of $f \in L^2(\varpi_\mu , \BB^d )$ is defined by
\begin{equation}\label{Fourier-coef}
f(x) = \sum_{n=0}^{\infty} \sum_{j=0}^{\lfloor \frac{n}{2} \rfloor} \sum_{\nu}
     \widehat{f}_{j,\nu}^{n,\mu} P_{j,\nu}^{n,\mu}(x), \quad\hbox{with}\quad 
         \widehat{f}_{j,\nu}^{n,\mu} := \frac{1}{h_{j,n}^{\mu}}\la f, P_{j,\nu}^{n,\mu} \ra_\mu. 
\end{equation}
Let $\proj_n^\mu: L^2(\varpi_\mu, \BB^d) \mapsto \CV_n^d(\varpi_\mu)$ and $S_n^\mu:  L^2(\varpi_\mu, \BB^d) 
\mapsto \Pi_n^d$ denote the projection operator and the $n$-th partial sum operator, respectively. Then
\begin{equation} \label{proj-oper}
S_n^\mu f (x) = \sum_{m=0}^n \proj_m^\mu f (x) \quad\hbox{and} \quad \proj_m^\mu f (x) = \sum_{j=0}^{\lfloor\frac{m}{2}\rfloor}
\sum_{\nu} \widehat{f}_{j,\nu}^{m,\mu} P_{j,\nu}^{m,\mu}(x),
\end{equation}
By definition, $S_n^\mu f = f$ if $f \in \Pi_n^d$ and $\la f-S_n^\mu f, v \ra_\mu = 0$ for all $v \in \Pi_n^d$.

It turns out that the partial derivatives commute with the partial sum operators, a fact that plays an essential
role in our development below. 

\begin{lem} Let $\mu > -1$. Then
\begin{equation} \label{eq:comm1}
\partial_i S_n^\mu f = S_{n-1}^{\mu+1} (\partial_i f), \qquad 1 \le i \le d,
\end{equation}
and 
\begin{equation} \label{eq:comm2}
  D_{i,j} S_n^\mu f = S_n^{\mu} (D_{i,j} f), \qquad 1 \le i <j \le d. 
\end{equation}
\end{lem}

\begin{proof}
By its definition, $f - S_n^\mu f = \sum_{m=n+1}^{\infty} \proj_m^\mu f$ and 
$\proj_m^\mu f \in \CV_m^d(\varpi_\mu)$. From Proposition~\ref{prop:ballOP},
we can easily deduce $\partial_i \proj_m^\mu f \in \CV_{m-1}^d(\varpi_{\mu+1})$, so that 
$\la \partial_i (f - S_n^\mu f), P \ra_{\mu+1} = 0$ for all $P \in \Pi_{n}^d$. Consequently, 
$S_{n-1}^{\mu+1} (\partial_i f - \partial_i S_n^\mu f) = 0$. Since $S_{n-1}^{\mu+1}$ reproduces polynomials of 
degree at most $n - 1$, $S_{n-1}^{\mu+1} (\partial_i S_n^\mu f) = \partial_i S_n^\mu f$, which implies that
$$
0 = S_{n-1}^{\mu+1} (\partial_i f - \partial_i S_n^\mu f) = 
S_{n-1}^{\mu+1} (\partial_i f) - \partial_i S_n^\mu f.
$$
This proves \eqref{eq:comm1}, which implies the first identity in \eqref{eq:comm2}. Now, $D_{i,j}$ maps 
$\CH_n^d$ to itself, which implies that $D_{i,j} \proj_m^\mu f \in \CV_{m}^d(\varpi_{\mu})$ and, as a result, 
that \eqref{eq:comm2} can be established similarly as \eqref{eq:comm1}.
\end{proof}

The relations in the above lemma pass down to the Fourier coefficients. 

\begin{prop} 
Let $f \in \mathcal{W}_2^{2}(\varpi_\mu , \mathbb{B}^d )$ and let $\widehat{f}_{j,\nu}^{n,\mu}$ be as defined 
in \eqref{Fourier-coef}. Then
\begin{align} \label{Delta-Fourier}
\widehat{\Delta f}_{j,\nu}^{n-2,\mu+2} = \k_{n-j-1}^\mu \widehat{f}_{j+1,\nu}^{n,\mu} \quad \hbox{and}\quad
\widehat{\Delta_0 f}_{j,\nu}^{n,\mu} = \l_{n-2j} \widehat{f}_{j,\nu}^{n,\mu},
\end{align}
where $0 \le j \le (n-2)/2$ in the first identity and $0 \le j \le n/2$ in the second identity. 
\end{prop}

\begin{proof}
From $\proj_n^\mu f = S_n^\mu f - S_{n-1}^\mu f$ and \eqref{eq:comm1}, we obtain
$\Delta \proj_n^\mu f = \proj_{n-2}^{\mu+2} \Delta f$. By \eqref{proj-oper} and \eqref{laplace_P},
\begin{align*}
\Delta \proj_n^\mu f(x)  = & \sum_{j=1}^{\lfloor\frac{n}{2}\rfloor} \sum_{\nu}  \widehat{f}_{j,\nu}^{n,\mu} 
           \k_{n-j}^\mu P_{j-1,\nu}^{n-2,\mu+2} (x) \\
       = &\sum_{j=0}^{\lfloor\frac{n-2}{2}\rfloor} \sum_{\nu}  \widehat{f}_{j+1,\nu}^{n,\mu}
           \k_{n-j-1}^\mu P_{j,\nu}^{n-2,\mu+2} (x). 
\end{align*}
Hence, by $\Delta \proj_n^\mu f = \proj_{n-2}^{\mu+2} \Delta f$, the first identity of \eqref{Delta-Fourier} follows. 
The second identity of \eqref{Delta-Fourier} follows similarly, using \eqref{laplace_P} and $\Delta_0 S_n^\mu f =
S_n^\mu \Delta_0 f$, the latter follows from \eqref{eq:comm2} and \eqref{eq:Delta0=}.
\end{proof}
 
We are now ready to prove Theorem \ref{thm:1.1}.  

\bigskip
\noindent
{\it Proof of Theorem \ref{thm:1.1}}.
We start with Parseval's identity, 
$$
E_n(f)_{\mu}^2 = \|f - S_n^\mu f \|_{\mu}^2 = \sum_{m=n+1}^{\infty}\sum_{j=0}^{\lfloor\frac{m}{2}\rfloor}
\sum_{\nu}  \left|\widehat{f}_{j,\nu}^{m,\mu}\right|^2 h_{j,m}^{\mu} = \Sigma_1 + \Sigma_2,
$$ 
where we split the sum as 
\begin{align*}
\Sigma_1 = \sum_{m=n+1}^{\infty}\sum_{j=\lfloor\frac{m}{4}\rfloor}^{\lfloor\frac{m}{2}\rfloor}
\sum_{\nu} \left|\widehat{f}_{j,\nu}^{m,\mu}\right|^2 h_{j,m}^{\mu}
     \quad\hbox{and}\quad
\Sigma_2 = \sum_{m=n+1}^{\infty}\sum_{j=0}^{\lfloor\frac{m}{4}\rfloor-1}
\sum_{\nu} \left|\widehat{f}_{j,\nu}^{m,\mu}\right|^2 h_{j,m}^{\mu}.  
\end{align*}

We estimate $\Sigma_1$ first. Iterating the first identity in \eqref{Delta-Fourier}, we obtain
\begin{equation}\label{eq:F-coeff1}
\left| \widehat{f}_{j,\nu}^{\mu,m}\right|^2 = \prod_{i=1}^s \left(\k_{m-j-i+1}^{\mu+2i}\right)^{-2} 
                      \left| \widehat{\Delta^s f}_{j-s,\nu}^{m-2s,\mu+2s}\right|^2 
    \sim m^{-4 s}  \left| \widehat{\Delta^s f}_{j-s,\nu}^{m-2s,\mu+2s}\right|^2 
\end{equation}
for $\lfloor\frac{m}{4}\rfloor \le j \le \lfloor\frac{m}{2}\rfloor$. Furthermore, by \eqref{eq:Hjn-mu}, it is easy to verify that 
\begin{equation}\label{eq:h-ratio}
\frac{h_{j,m}^{\mu}}{h_{j-s,m-2s}^{\mu+2s}} = 
  \frac{(\mu + 1)_{2s} (m-j-s+\f{d}{2})_s (\mu + m-j + \frac{d+2}{2})_s}
  {(\mu + \frac{d+2}{2})_{2s} (j-s+1)_s (j + \mu+1)_s},
\end{equation}
which is bounded by a constant, independent of $m$, when $j \sim m$. Consequently, it follows that
$$
\Sigma_1 \leq c  \sum_{m=n+1}^{\infty}\sum_{j=\lfloor\frac{m}{4}\rfloor}^{\lfloor\frac{m}{2}\rfloor}
\sum_{\nu}  m^{-4 s}  \left| \widehat{\Delta^s f}_{j-s,\nu}^{m-2s,\mu+2s}\right|^2 h_{j-s,m-2s}^{\mu+2s}
  \le \frac{c}{n^{4s}}E_{n-2s}(\Delta^s f)_{\mu+2s}^2.
$$
Next, we estimate $\Sigma_2$. Iterating the second identity in \eqref{Delta-Fourier}, we obtain
\begin{equation}\label{eq:F-coeff2}
\left|\widehat{f}_{j,\nu}^{m,\mu}\right|^2 =  \left( \lambda_{m-2j}\right)^{-2s} \left|\widehat{\Delta_0^s f}_{j,\nu}^{m,\mu}\right|^2 
  \sim m^{-4 s}  \left|\widehat{\Delta_0^s f}_{j,\nu}^{m,\mu}\right|^2 
\end{equation}
for $0 \le j \le \lfloor\frac{m}{4}\rfloor$. Consequently, it follows that 
$$
\Sigma_2 \leq c  \sum_{m=n+1}^{\infty}\sum_{j=0}^{\lfloor\frac{m}{4}\rfloor-1}
\sum_{\nu}  m^{-4 s}  \left|\widehat{\Delta_0^s f}_{j,\nu}^{m,\mu}\right|^2 
  h_{j,m}^{\mu} 
  \le \frac{c}{n^{4s}}E_{n}(\Delta_0^s f)_{\mu}^2.
$$
Putting these two estimates together completes the proof of the theorem.
\qed

\medskip

In the above proof, we do not need to specify the basis of spherical harmonics in the definition of $P_{j,\nu}^{n,\mu}$. 
The proof of Theorem \ref{thm:1.3} is far more complicated, for which we do need to specify the basis. Thus, in the
rest of this section, we shall choose the basis of $\CH_{n-2j}^d$ as the one in \eqref{spher-basis} and we shall adopt
the convention as in the discussion right before Proposition \ref{prop:ballOP}. 

We need two estimates on the Fourier coefficients. 

\begin{prop} 
Let $f \in \mathcal{W}_2^1(\varpi_\mu , \BB^d)$ and let $\widehat{f}_{j,\nu}^{n,\mu}$ be as defined in \eqref{Fourier-coef}. 
Then
\begin{equation} \label{eq:Fcoeff-est2}
  \left |\wh f_{\ell,\nu}^{m,\mu}\right | \le \frac{c}{m} \sum_{\eta} \sum_{i=1}^d 
     \left( \left | \wh{\partial_i f}_{\ell-1,\eta}^{m-1,\mu+1} \right|+\left | \wh{\partial_i f}_{\ell,\eta}^{m-1,\mu+1} \right|\right),
\end{equation}
where the sum over $\eta$ consists of at most $2^{d-1}$ many terms. 
\end{prop}

\begin{proof}
In the one hand, by \eqref{eq:nabla-P} and the definition of $\proj_n^\mu f$, we see that 
\begin{align*}
  h_{\ell,m}^\mu(\nabla) \wh f_{\ell,\nu}^{m,\mu}  = b_\mu \int_{\BB^d} \nabla \proj_m^\mu f (x)\cdot \nabla P_{\ell,\nu}^{m,\mu}(x) 
       \varpi_{\mu+1}(x)dx.
\end{align*}
On the other hand, 
\begin{align*}
&  \sum_{i=1}^d b_\mu \int_{\BB^d} \proj_{m-1}^{\mu+1} \partial_i f (x) \partial_i P_{\ell,\nu}^{m,\mu}(x) 
       \varpi_{\mu+1}(x)dx \\
& \qquad\qquad = \frac{b_\mu}{b_{\mu+1}} \, \sum_{j=0}^{\lfloor \frac{m}{2} \rfloor} \sum_\eta \sum_{i=1}^d \wh {\partial_i f}_{j,\eta}^{m-1,\mu+1}
          \la P_{j,\eta}^{m-1,\mu+1}, \partial_i P_{\ell,\nu}^{m,\mu}\ra_{\mu+1}. 
\end{align*}
Hence, taking inner product of both sides of the identity \eqref{eq:comm1} with $\partial_i P_{\ell,\nu}^{m,\mu}$ in the inner
product of $L^2(\varpi_{\mu+1}, \BB^d)$, we obtain 
\begin{equation}\label{eq:fourier-coeff2}
 h_{\ell,m}^\mu(\nabla) \wh f_{\ell,\nu}^{m,\mu} = \frac{b_\mu}{b_{\mu+1}} \,   
 \sum_{j=0}^{\lfloor \frac{m}{2} \rfloor} \sum_\eta \sum_{i=1}^d \wh {\partial_i f}_{j,\eta}^{m-1,\mu+1}
          \la P_{j,\eta}^{m-1,\mu+1}, \partial_i P_{\ell,\nu}^{m,\mu}\ra_{\mu+1}, 
\end{equation}
where the number of terms in the sum over $\eta$ consists of at most $2^{d-1}$ terms by Proposition \ref{prop:ballOP}.
We now estimate the coefficients in the right hand side. 
 
By Proposition \ref{prop:ballOP}, the coefficients $ \la P_{j,\eta}^{m-1,\mu+1}, \partial_i P_{\ell,\nu}^{m,\mu}\ra_{\mu+1}$ 
in \eqref{eq:fourier-coeff2} are nonzero only if $j = \ell$ or $j=\ell-1$. Now, by the Cauchy-Schwarz inequality and 
\eqref{eq:nabla-P},
\begin{align*}
  \left |\frac{b_\mu}{b_{\mu+1}}\,\la P_{j,\eta}^{m-1,\mu+1}, \partial_i P_{\ell,\nu}^{m,\mu}\ra_{\mu+1}\right |^2 
  \le & \frac{b_\mu^2}{b_{\mu+1}^2}
    \left \| P_{j,\eta}^{m-1,\mu+1}\right \|^2_{\mu+1} \left \| \partial_i P_{\ell,\nu}^{m,\mu} \right\|_{\mu+1}^2 \\
    \le &\, \frac{b_\mu}{b_{\mu+1}} h_{j,m-1}^{\mu+1} h_{\ell,m}^\mu(\nabla) \le \frac{c}{m^2}  h_{j,m}^\mu(\nabla)^2 
\end{align*}
where the last step is deduced using the explicit formula of the two norms in \eqref{eq:Hjn-mu} and \eqref{eq:nabla-P}, 
and the fact that $j = \ell$ or $j = \ell-1$. Consequently, since the number of terms in the sum over $\eta$ is independent
of $m$, \eqref{eq:Fcoeff-est2} follows from \eqref{eq:fourier-coeff2}. 
\end{proof}

\begin{prop} 
Let $f \in \mathcal{W}_2^1(\varpi_\mu , \BB^d)$ and let $\widehat{f}_{j,\nu}^{n,\mu}$ be as defined in \eqref{Fourier-coef}. 
Then
\begin{equation} \label{eq:Fcoeff-est1}
  \left |\wh f_{\ell,\nu}^{m,\mu}\right | \le \frac{1}{\sqrt{(m-2\ell)(m-2\ell+d-2)}} \sum_\eta \sum_{i<j}  
      \left | \wh{D_{i,j} f}_{\ell,\eta}^{m,\mu} \right|
\end{equation}
where the sum over $\eta$ consists of finitely many terms independent of $m$. 
\end{prop}

\begin{proof}
In the one hand, using \eqref{eq:Dij-P} and the definition of $\proj_n^\mu f$, we obtain
$$
  b_\mu \int_{\BB^d} \sum_{i<j} D_{i,j} \proj_m^\mu f(x) D_{i,j} P_{\ell,\nu}^{m,\mu} 
  \varpi_\mu(x) dx =  (m-2\ell)(m-2\ell+d-2) h_{\ell,m}^\mu \wh f_{\ell,\nu}^{m,\mu}.  
$$
On the other hand, since $D_{i,j}$ maps $\CH_m^d$ to itself, examining the radial part of the orthogonality of
$P_{\ell,\nu}^{m,\mu}$ shows that $\la P_{k,\nu}^{m,\mu}, D_{i,j} P_{\ell,\nu}^{m,\mu}\ra_\mu =0$ if $k\ne \ell$, 
which implies that 
\begin{align*}
b_\mu \int_{\BB^d} \sum_{i<j} & \proj_m^{\mu} (D_{i,j} f; x) D_{i,j} P_{\ell,\nu}^{m,\mu} \varpi_\mu(x) dx  \\
 \qquad  & =  \sum_\eta \sum_{i<j}
 b_\mu \int_{\BB^d}  P_{\ell,\eta}^{m,\mu}(x)  D_{i,j} P_{\ell,\nu}^{m,\mu} (x) \varpi_\mu(x) dx  \, \wh{D_{i,j} f}_{\eta,\nu}^{m,\mu}.
\end{align*}
Since $D_{i,j} \proj_m^\mu f = \proj_m^\mu D_{i,j} f$ by \eqref{eq:comm2}, comparing the above two expressions, we obtain 
\begin{align}  \label{eq:f=Df-coeff}
 (m-2\ell)(m-2\ell+d-2) h_{\ell,m}^\mu \wh f_{\ell,\nu}^{m,\mu}  =  \sum_\eta \sum_{i<j}
\la P_{\ell,\eta}^{m,\mu}, D_{i,j} P_{\ell,\nu}^{m,\mu}\ra_\mu \, \wh{D_{i,j} f}_{\ell,\eta}^{m,\mu}.
\end{align}
By Proposition \ref{prop:ballOP}, the sum over $\eta$ consists of at most $2^{2d-1}$ terms. Now, by the Cauchy-Swartz 
inequality and \eqref{eq:Dij-P},
\begin{align*}
\left | \la P_{\ell,\eta}^{m,\mu}, D_{i,j} P_{\ell,\nu}^{m,\mu}\ra_\mu \right|^2 
& \le h_{\ell,m}^\mu  b_\mu \int_{\BB^d} \left| D_{i,j} P_{\ell,\nu}^{m,\mu} \right|^2 \varpi_\mu(x) dx \\
& \le h_{\ell,m}^\mu  b_\mu \int_{\BB^d} \sum_{i< j} \left| D_{i,j} P_{\ell,\nu}^{m,\mu} \right|^2 \varpi_\mu(x) dx \\
& = (m-2\ell)(m-2 \ell+d-2) (h_{\ell,m}^\mu )^2. 
\end{align*}
Applying this inequality on \eqref{eq:f=Df-coeff} and using the fact the number of terms in the sum over $\eta$ is 
independent of $m$, we complete the proof. 
\end{proof}

\bigskip
\noindent
{\it Proof of Theorem \ref{thm:1.3}}. As in the proof of Theorem \ref{thm:1.1}, we write
$$
E_n(f)_{\mu}^2 = \|f - S_n^\mu f \|_{\mu}^2 = \sum_{m=n+1}^{\infty}\sum_{j=0}^{\lfloor\frac{m}{2}\rfloor}
\sum_{\nu} \left| \widehat{f}_{j,\nu}^{m,\mu}\right|^2 h_{j,m}^{\mu} = \Sigma_1 + \Sigma_2,
$$ 
using the same split up. We work with $\Sigma_1$ first. For $\lfloor\frac{m}{4}\rfloor \le j \le \lfloor\frac{m}{2}\rfloor$, we
obtain by \eqref{eq:F-coeff1} and \eqref{eq:Fcoeff-est2} that 
\begin{align*}
\left|\wh f_{j,\nu}^{m,\mu}\right|^2 \le c m^{-4 s -2} \sum_\eta 
     \sum_{i=1}^d \left( \left |\wh{\partial_i \Delta^s f}_{j-s-1,\eta}^{m-2s-1,\mu+2s+1}\right|^2
       + \left |\wh{\partial_i \Delta^s f}_{j-s,\eta}^{m-2s-1,\mu+2s+1}\right|^2\right).
\end{align*}
Moreover, as in \eqref{eq:h-ratio}, it is not difficult to see, using \eqref{eq:Hjn-mu}, that 
\begin{equation} \label{eq:h-sim-h}
 \frac{h_{j,m}^\mu}{ h_{j-s-1,m-2s-1}^{\mu+2s+1}} \sim 
     \frac{h_{j,m}^\mu}{ h_{j-s,m-2s-1}^{\mu+2s+1}} \sim \frac{(m-j+1)^{2s+1}}{(j+1)^{2s+1}}
\end{equation}
which shows, in particular, that $h_{j,m}^\mu / h_{j-s-1,m-2s-1}^{\mu+2s+1} \le c $
and  $h_{j,m}^\mu / h_{j-s,m-2s-1}^{\mu+2s+1} \le c $ for $\lfloor\frac{m}{4}\rfloor \le j \le \lfloor\frac{m}{2}\rfloor$. 
Furthermore, since the number of terms in the summation over $\eta$ is independent of $m$, we see that
$$
 \sum_\nu \sum_\eta  \left|\wh{\partial_i \Delta^s f}_{j-s,\eta}^{m-2s-1,\mu+2s+1}\right|^2 \le c \sum_{1 \le \tau \le a_{m-1-2j}^d}
      \left|\wh{\partial_i \Delta^s f}_{j-s,\tau}^{m-2s-1,\mu+2s+1}\right|^2.
$$
Consequently, 
\begin{align*}
\Sigma_1 \leq & \, c  \sum_{m=n+1}^{\infty}\sum_{j=\lfloor\frac{m}{4}\rfloor}^{\lfloor\frac{m}{2}\rfloor}
  \frac{1}{m^{2(2s+1)}} \sum_\nu \sum_\eta \sum_{i=1}^d \left|\wh{\partial_i \Delta^s f}_{j-s,\eta}^{m-2s-1,\mu+2s+1}\right|^2 
       h_{j-s,m-2s-1}^{\mu+2s+1} \\
    \leq & \,  \frac{c}{n^{2(2s+1)}}   \sum_{i=1}^d  \sum_{m=n+1}^{\infty}\sum_{j=\lfloor\frac{m}{4}\rfloor}^{\lfloor\frac{m}{2}\rfloor}
 \sum_\tau \left|\wh{\partial_i \Delta^s f}_{j-s,\tau}^{m-2s-1,\mu+2s+1}\right|^2 
       h_{j-s,m-2s-1}^{\mu+2s+1} \\    
 \le\, & \frac{c}{n^{2(2s+1)}} \sum_{i=1}^d E_{n-2s-1}(\partial_i \Delta^s f)_{\mu+2s+1}^2.   
\end{align*}
Next we estimate $\Sigma_2$. For $0 \le j \le \lfloor\frac{m}{4}\rfloor$, we obtain by \eqref{eq:F-coeff2} and 
\eqref{eq:Fcoeff-est1} that 
\begin{align*}
\left|\wh f_{j,\nu}^{m,\mu}\right|^2 \le c m^{-4 s -2} \sum_\eta \sum_{i<j} |\wh{D_{i,j} \Delta_0^s f}_{j,\eta}^{m,\mu}|^2.
\end{align*}
Hence, it follows that 
\begin{align*}
\Sigma_2 &  \le  c \sum_{i< j} \sum_{m=n+1}^{\infty} \frac{1}{m^{4 s+2}} \sum_{j=0}^{\lfloor\frac{m}{4}\rfloor}
      \sum_\nu \sum_\eta  |\wh{D_{i,j} \Delta_0^s f}_{j,\eta}^{m,\mu}|^2 h_{j,m}^{\mu} \\
       & \le \frac{c}{n^{2(2s+1)}} \sum_{i< j} E_n(D_{i,j} \Delta_0^s f)_{\mu}^2,
\end{align*}
where we have used again that the number of terms in the summation over $\eta$ is independent of $m$.
Putting the two estimates together completes the proof. 
\qed

\medskip

\begin{rem}\label{rem:final}
As we see from \eqref{eq:h-sim-h}, the restriction $j \sim m$ is essential for the two ratios to be bounded by a 
constant. This shows that the estimate of $\Sigma_1$ in the proof does not work for small $j$, which gives a strong indication
that the estimate \eqref{eq:thm1.3}, in particular \eqref{main-first}, is unlikely to hold without the second term. This is
the case especially when $f$ is a spherical function, that is, $f(x) = f_0(x/\|x\|)$, for which we need $j=0$ since 
$\wh f_{j,\nu}^n =0$ if $j \ne 0$. 
\end{rem}

\medskip\noindent
{\bf Acknowledgement}: The authors thank anonymous referees for their helpful comments. 

\medskip

\section*{Appendix: The derivatives of spherical harmonics for $d=2,3$}
\setcounter{equation}{0}

For the spherical harmonics in \eqref{spher-basis}, we can derive explicit expressions of their partial derivatives, written
as a sum of the same basis but one degree lower. We state these explicit expressions for the cases $d=2$ and $d=3$
here. 

For $d =2$,  $\dim \mathcal{H}_n^2 = 2$ for $n \ge 1$. In polar coordinates, $(x_1, x_2) = (r \cos \theta, r \sin \theta)$, 
of $\mathbb{R}^2$, an orthogonal basis of $\mathcal{H}_n^2$ is given by
$$
Y^{(1)}_n (x) = r^n \cos n \theta, \quad Y^{(2)}_n (x) = r^n \sin n \theta,
$$
which agrees with \eqref{eq:harm-d=2}. A simple computation gives the following: 

\begin{proposition}
For $n>0$ we have
\begin{align*}
&\partial_1 Y^{(1)}_n (x) = n Y^{(1)}_{n-1} (x), \!\!\!\!\!\!\!\!\!\!\!\!&\partial_1 Y^{(2)}_n (x) = n Y^{(2)}_{n-1} (x) \\  
&\partial_2 Y^{(1)}_n (x) = -n Y^{(2)}_{n-1} (x), \!\!\!\!\!\!\!\!\!\!\!\!&\partial_2 Y^{(2)}_n (x) = n Y^{(1)}_{n-1} (x)
\end{align*}
\end{proposition}

\bigskip

For $d=3$, the space $\mathcal{H}_n^3$ of spherical harmonics of degree $n$ has dimension $2n + 1$. 
In spherical polar coordinates of $\mathbb{R}^3$, 
\begin{align*}
\begin{cases} x_1 = r \sin \theta \sin \phi,\\
x_2 = r \sin \theta \cos \phi,\\
x_3 = r \cos \theta,
\end{cases}
\quad 0 \le \theta \le \pi, \, 0 \le \phi < 2 \pi, \, r > 0,
\end{align*}
a mutually orthogonal basis of $\mathcal{H}_n^3$ is given by
\begin{align*}
Y_{k,1}^n(x) &= r^n (\sin \theta )^k C_{n-k}^{k+\frac{1}{2}}(\cos \theta ) \cos k \phi, \quad 0 \le k \le n,\\
Y_{k,2}^n(x) &= r^n (\sin \theta )^k C_{n-k}^{k+\frac{1}{2}}(\cos \theta ) \sin k \phi, \quad 1 \le k \le n.
\end{align*}
They are homogeneous polynomials in $x$ and agree with \eqref{spher-basis} when rewriting as 
\begin{align*}
Y_{k,1}^n(x) &= r^{n-k} C_{n-k}^{k+\frac{1}{2}}\left(\frac{x_3}{r}\right) \rho^k T_k\left(\frac{x_2}{\rho}\right), \quad 0 \le k \le n,\\
Y_{k,2}^n(x) &= r^{n-k} C_{n-k}^{k+\frac{1}{2}}\left(\frac{x_3}{r}\right) \rho^{k-1} x_1 U_{k-1}\left(\frac{x_2}{\rho}\right), \quad 1 \le k \le n,
\end{align*}
with $\rho = \sqrt{x_1^2 + x_2^2}$ and $r = \sqrt{x_1^2 + x_2^2 + x_3^2}$. In the next proposition, we will also define 
$Y_{k,1}^n(x) = Y_{k,2}^n(x) =0$ if $k< 0$ or $k > n$. 
 
\begin{proposition}
For $k = 0, 1, \ldots, n$, 
\begin{align*}
\partial_1 Y_{k,1}^n(x) &= - \frac{(n+k)(n+k-1)}{2(2k-1)} Y_{k-1,2}^{n-1}(x) - \left(k+\frac{1}{2}\right) Y_{k+1,2}^{n-1}(x),\\  
\partial_2 Y_{k,1}^n(x) &= \frac{(n+k)(n+k-1)}{2(2k-1)} Y_{k-1,1}^{n-1}(x) - \left(k+\frac{1}{2}\right) Y_{k+1,1}^{n-1}(x),\\
\partial_3 Y_{k,1}^n(x) &= (n+k) Y_{k,1}^{n-1}(x).
\end{align*}
For $k = 1, 2, \ldots, n$, 
\begin{align*}
\partial_1 Y_{k,2}^n(x) &= \frac{(n+k)(n+k-1)}{2(2k-1)} Y_{k-1,1}^{n-1}(x) + \left(k+\frac{1}{2}\right) Y_{k+1,1}^{n-1}(x),\\  
\partial_2 Y_{k,2}^n(x) &= \frac{(n+k)(n+k-1)}{2(2k-1)} Y_{k-1,2}^{n-1}(x) - \left(k+\frac{1}{2}\right) Y_{k+1,2}^{n-1}(x),\\
\partial_3 Y_{k,2}^n(x) &= (n+k) Y_{k,2}^{n-1}(x).
\end{align*}
\end{proposition}
\end{document}